\documentclass[a4paper,11pt, reqno]{amsart}

\usepackage[margin=2cm]{geometry}
\usepackage[utf8]{inputenc}
\usepackage[T1]{fontenc}
\usepackage{lmodern}
\usepackage[french,english]{babel}
\usepackage{amsthm, amssymb, amsmath, amsfonts, mathrsfs, dsfont, esint, textcomp, bm}
\usepackage[colorlinks=true, pdfstartview=FitV, linkcolor=blue, citecolor=blue, urlcolor=blue,pagebackref=false]{hyperref}
\usepackage{mathtools}
\usepackage{tikz}
\usepackage{enumitem, imakeidx}
\usetikzlibrary{patterns}

\usepackage{microtype}

\definecolor{darkblue}{rgb}{0,0,0.7} 
\definecolor{darkred}{rgb}{0.9,0.1,0.1}
\definecolor{darkgreen}{rgb}{0,0.5,0}

\newtheorem{thm}{Theorem}[section]

\newtheorem{lem}[thm]{Lemma}
\newtheorem{cor}[thm]{Corollary}
\theoremstyle{remark}
\newtheorem{rem}[thm]{Remark}
\theoremstyle{definition}

\renewcommand{\leq}{\leqslant}
\renewcommand{\geq}{\geqslant}

\renewcommand{\subset}{\subseteq}

\newcommand{\B}{\mathbb{B}}

\newcommand{\ah}{A_{\mathsf{h}}}

\newcommand{\uh}{u_{\mathsf{h}}}
\newcommand{\F}{\mathcal{F}}

\newcommand{\N}{\mathbb{N}}

\newcommand{\1}{\mathbf{1}}
\newcommand{\R}{\mathbb{R}}
\newcommand{\C}{\mathbb{C}}
\newcommand{\Q}{\mathbb{Q}}
\newcommand{\Z}{\mathbb{Z}}

\renewcommand{\P}{\mathbb{P}}

\newcommand{\eps}{\varepsilon}

\newcommand{\Rd}{\R^d}

\newcommand{\I}{\mathcal{I}}

\mathtoolsset{showonlyrefs}

\title[]{On the homogenization of random stationary elliptic operators in divergence form}
\author{Arianna Giunti, Juan J. L. Vel\'azquez}

\begin{document}

\maketitle

\begin{abstract}
In this note we comment on the homogenization of a random elliptic operator in divergence form $-\nabla \cdot a\nabla$, where the coefficient field $a$ 
is distributed according to a stationary, but not necessarily ergodic, probability measure $\P$. We generalize the well-known case for $\P$ stationary and ergodic 
by showing that the operator $-\nabla \cdot a(\frac{\cdot}{\eps})\nabla$ almost surely homogenizes to a constant-coefficient, random operator $-\nabla \cdot A_h\nabla$.
Furthermore, we use a disintegration formula for $\P$ with respect to a family of ergodic and stationary probability measures to show that the law of $A_h$ may be obtained 
by using the standard homogenization results on each probability measure of the previous family. We finally provide a more explicit formula for $A_h$ in the
case of coefficient fields which are a function of a stationary Gaussian field.

\end{abstract}


\section{Introduction}
This note provides a remark on the homogenization of random elliptic operators in divergence form $-\nabla \cdot a \nabla$, where the
coefficient field $a : \Rd \to \R^{d \times d}$ is symmetric, uniformly elliptic and distributed according to a probability measure $\P$ which is invariant with respect
to the translations $a(\cdot + x)$, $x\in \Rd$. We do not assume that $\P$ is ergodic.

\bigskip

In the case of stationary and ergodic measures, it is well-established that for $\P$-almost every realization of the coefficient field $a$, for $\eps \downarrow 0^+$ the
rescaled operator $-\nabla \cdot a(\frac{\cdot}{\eps}) \nabla$ homogenizes to $-\nabla \cdot A_{\mathrm{h}} \nabla$. The homogenized coefficient $A_{\mathrm{h}}$ is constant, 
deterministic and satisfies the same ellipticity bounds. Qualitative stochastic homogenization, {namely the convergence of solutions $u_\eps$ associated to 
$-\nabla \cdot a(\frac{\cdot}{\eps}) \nabla$ to the solution $\uh$ associated to $-\nabla \cdot A_{\mathrm{h}} \nabla$}, has been obtained in \cite{Koslov79} and
\cite{PapVar}; in the last two decades, a large literature has been developed to upgrade these results into quantitative estimates on the convergence of $u_\eps$
to $u_h$ (e.g. \cite{ArmstrongKuusiMourrat, ArmstrongMourrat, ArmstrongSmart, BellaGiuntiOttoPCMI, GNO4, GNO}). This has led, in addition, to an exhaustive understanding of the fluctuations structure of $u_\eps$ and of other meaningful
quantities related to the random operator $-\nabla \cdot a \nabla$ \cite{ArmstrongKuusiMourrat, DuerinckxGloriaOtto, DuerinckxGloriaOtto2,  gumourrat-fluctuations}. 

\bigskip

As is well-known in classical stochastic homogenization \cite{Koslov79, PapVar}, namely when the measure $\P$ is stationary and ergodic, the homogenized
matrix $\ah$ may be identified with the large-scale limit of the spacial {averages of suitable stationary random fields, namely the flux of the correctors $a (e_i + \nabla \phi_i)$  (see \eqref{corrector.eq} and \eqref{ergodic.homogenized.coefficient}). We refer to \cite{BellaFehrmanFischerOtto,  BellaGiuntiOtto2nd, FischerOtto, Gu_highorder, MourratOtto} for an extensive study of the correctors and their properties.
This characterization of $\ah$ allows to appeal to Birkhoff's ergodic theorem (see e.g. \cite{kre}) and infer that $\ah$ is well-defined for $\P$-almost every realization
of $a$ and, by the ergodicity assumption, that $\ah$ does not depend on $a$. It is thus intuitive to expect that, if only the ergodicity assumption on $\P$ fails, one still
obtains a homogenization result for $-\nabla \cdot a(\frac{\cdot}{\eps}) \nabla$, this time with the homogenized coefficient $A_{\textrm{h}}= \ah(a)$ being a
random matrix. The first result contained in this note gives a rigorous derivation of this argument (see Theorem \ref{t.main}).

\smallskip

If we denote by $\Omega$ the space of realizations of $a$ and by $(\Omega, \F, \P)$ the associated probability space, Birkhoff's ergodic theorem also implies 
that the random matrix $\ah$ is measurable with respect to the $\sigma$-algebra $\I \subset \F$ generated by the subsets of $\Omega$ which are translation invariant. 
In other words, in the case of stationary measures, the homogenization process does not remove the randomness from $\ah$ but leads nonetheless to a 
reduction in its complexity. We give a further result in this direction by relying on some techniques coming from statistics and dynamical systems that allow to write $\P$ as a disintegration with respect to a family of stationary and ergodic probability
measures. More precisely, if $\P$ is a stationary probability measure on $(\Omega, \F)$, then for all $B \in \F$
\begin{align}\label{decomposition.P}
\P(B) = \int_{\Omega_0} \P_{\xi}(B) \tilde\P(d\xi), 
\end{align}
where $(\Omega_0, \I_0)$ is a measurable space, $\{ P_\xi\}_{\xi \in \Omega_0}$ is a family of ergodic and stationary probability measures on $(\Omega, \F)$, and 
$\tilde \P$ is a probability measure on $(\Omega_0, \I)$ (see \cite{blum1972, Gray.Probability2001} and Lemma \ref{Ergodic.disintegration}). More precisely, the set
$\Omega_0$ is obtained as a quotient $\Omega/ \sim$ with respect to a suitable equivalence relation $\sim$, and $\I_0$ is isomorphic to the $\sigma$-algebra $\I$ of the sets of $\Omega$
invariant under translations. Using \eqref{decomposition.P}, we show that $A_{\textrm{h}}$ may be identified with a random variable on $(\Omega_0, \I_0, \tilde\P)$ and
that we may simply obtain the realization $\ah(\xi)$ for $\xi \in \Omega_0$ by appealing to the standard homogenization result for 
the ergodic and stationary measure $(\Omega, \F, \P_\xi)$ (Corollary \ref{cor.abstract}). 

\smallskip

As an application of the previous result, we study the case of coefficient fields $a$ which are Gaussian related, namely when $a$ is a function of a 
stationary Gaussian field. The Gaussian setting allows to obtain a more explicit disintegration of $\P$ and an explicit formula for $\Omega_0$ and $\tilde \P$ which 
characterize the law of the random matrix $A_{\textrm{h}}$ (Corollary \ref{c.gaussian}). 

\bigskip

We conclude this introduction by mentioning that in \cite{DalMaso.nonergodic} a homogenization result for random \textit{free-discontinuity functionals} has
also been obtained in the setting of stationary measures which are not assumed to be ergodic. 

\section{Notation and abstract result}
Let $(\Omega, \F, \P)$ be a probability space equipped with a group of transformations $\{ \tau_x : \Omega \rightarrow \Omega \}_{x\in \Rd}$, $d \geq 2$, with respect to 
which the measure $\P$ is stationary, i.e. 
\begin{align}\label{stationarity}
\P \circ \tau_x = \P \ \ \ \forall x\in \Rd.
\end{align}
We assume that $\F$ is countably generated and that for all $B \in \F$ 
\begin{align}\label{stochastic.continuity}
\lim_{|x| \downarrow 0} \int | 1_B( \tau_x \omega) - 1_B(\omega)| \P(d\omega)= 0.
\end{align}
Form this it follows that the joint map $\tau: \Omega \times \Rd \to \Omega$, $\tau(x, \omega) = \tau_x \omega$ is measurable with respect to
the tensor $\sigma$-algebra of  $\F$ and of the Lebesgue measurable sets of $\Rd$. 

\smallskip

In addition, the assumptions on $\F$ also imply that the spaces $L^p(\Omega, \F, \P)$ are separable for all $1\leq p< +\infty$ and  that { the maps
$T_x : L^p(\Omega, \F, \P) \to L^p(\Omega, \F, \P)$, \ $T_x F:= F \circ \tau_x$ are strongly continuous for all $x\in \Rd$ and $1 \leq p < +\infty$.} 
For $F \in L^1(\Omega, \F, \P)$, we define
\begin{align}
 \langle \, F  \, \rangle := \int_\Omega F(\omega)  \, \P(d\omega).
\end{align}

\bigskip

Let $\mathcal{M}_{d, \textrm{sym}}$ denote the space of symmetric $d \times d$ real matrices and let $(\Omega, \F, \P)$ be as above. We 
define the random coefficient field $a$ as follows: Let $A: \Omega \to \mathcal{M}_{d, \textrm{sym}}$ be a (measurable, matrix-valued) random variable satisfying for
$0 < \lambda \leq \Lambda$ and $\P$-almost every $\omega \in \Rd$
\begin{align}\label{ellipticity}
\lambda|\xi|^2 \leq \xi \cdot A(\omega) \xi \leq \Lambda|\xi|^2 \ \ \ \ \forall \zeta \in \Rd.
\end{align}
We set 
\begin{align}\label{coefficient.field}
a: \Rd \times \Omega \rightarrow  \mathcal{M}_{d, \textrm{sym}}, \ \ \ a(\omega, x) = T_x A( \omega) = A( \tau_x \omega).
\end{align}
Thus, for $\P$-almost every $\omega \in \Omega$ the operator $-\nabla \cdot a \nabla$ is bounded and uniformly elliptic.

\smallskip

We emphasize that we do not require that $\P$ is ergodic: By Birkhoff's ergodic theorem \cite{kre} we have that for any $F \in L^1(\Omega, \F, \P)$ and $\P$-almost every $\omega \in \Omega$
\begin{align}\label{Birkhoff}
 \lim_{R \uparrow +\infty} \fint_{|x|< R} F(\tau_x \omega) \, dx = \langle F \, | \, \I \rangle,
\end{align}
where the right-hand side is the conditional expectation of $F$ with respect to the $\sigma$-algebra $\I \subset \F$ generated by the sets
$$
A \in \F, \ \ \ \tau_x A = A \ \ \forall x \in \Rd. 
$$

\smallskip

We recall that in the ergodic case, i.e. when $\I$ is trivial and the right-hand side of \eqref{Birkhoff} is given by $\langle \, F \, \rangle$, it is
well-known \cite{Koslov79, PapVar} that the operator $-\nabla \cdot a(\omega, \frac{\cdot}{\eps}) \nabla$   homogenizes $\P$-almost surely to the operator
$-\nabla \cdot \ah \nabla$. The matrix $\ah \in \mathcal{M}_{d, \textrm{sym}}$ is constant and deterministic and is given by the formula
\begin{align}\label{ergodic.homogenized.coefficient}
e_i \cdot \ah e_j = \langle (e_i + \nabla \phi_i(0) ) \cdot A ( e_j + \nabla \phi_j(0)) \rangle, \ \ \ \ i,j= 1, \cdots, d.
\end{align}
Here, for each $i = 1, \cdots, d$ the random fields $\phi_i(\omega, \cdot)\in H^1_{loc}(\Rd)$ are the \textit{first-order correctors} \cite{GNO, Koslov79, PapVar} satisfying for almost every
$\omega \in \Omega$
\begin{equation}\label{corrector.eq}
\begin{aligned}
-\nabla \cdot a(\omega, x) \nabla (\phi_i(\omega, x) + x_i) = 0 \ \ \ \ \text{ in $\Rd$,}\\
\lim_{R\uparrow +\infty} R^{-2} \fint_{|x|< R} | \phi_i(\omega, x) - \fint_{|y|<R} \phi_i(\omega, y) \, dy |^2 \, dx = 0. 
\end{aligned}
\end{equation}
Note that the functions $\phi_i(\omega, \cdot)$ are uniquely defined up to a random variable. 

\subsection{Abstract results}
\begin{thm}\label{t.main}
Let $(\Omega, \F, \P)$ be as above and let $a$ be as in \eqref{coefficient.field}. Then, there exists a random variable $\ah: \Omega \to \mathcal{M}_{d, \textrm{sym}}$ such
that for any bounded open set $D \subset \Rd$, $f \in H^{-1}(D)$ and almost every $\omega \in \Omega$, the solutions to the Dirichlet problem
\begin{align}\label{P.eps}
 \begin{cases}
  -\nabla \cdot a(\omega, \frac x \eps) \nabla u_\eps(x) = f(x)  \hspace{1cm} &\text{in $D$}\\
  u_\eps = 0 &\text{on $\partial D$}
 \end{cases}
\end{align}
converge weakly in $H^1_0(D)$ to the (random) solution of
\begin{align}\label{P.hom}
 \begin{cases}
  -\nabla \cdot \ah(\omega) \nabla \uh(\omega, x) = f(x) \hspace{1cm} &\text{in $D$}\\
  \uh(\omega, x) = 0 &\text{on $\partial D$.}
 \end{cases}
\end{align}
Moreover,
\begin{align}\label{hom.conditional}
e_i \cdot \ah(\omega) e_j = \langle (e_i + \nabla \phi_i(0))\cdot A (e_i + \nabla \phi_j(0)) \, | \, \I \rangle,
\end{align}
where each $\phi_i(\omega, \cdot)$, $i=1, \cdots, d$ satisfies \eqref{corrector.eq} with respect to the probability space $(\Omega, \F, \P)$.
\end{thm}

The term on the right-hand side of \eqref{hom.conditional} admits a further reformulation in terms of the ergodic decomposition for the measure $\P$. 
This is a standard result in the theory of asymptotically mean stationary processes (see, e.g., \cite{Gray.Probability2001}[Chapter 7, Theorem 7.4.1]):

\begin{lem}[\bf Ergodic decomposition]\label{Ergodic.disintegration}
{There exist} a family $\{  \P_\xi \}_{\xi \in \Omega_0}$ of ergodic and stationary probability measures on
$(\Omega, \F)$ and a probability space $(\Omega_0, \I_0, \tilde \P)$ such that the measure $\P$ admits the disintegration
\begin{align}\label{ergodic.disintegration}
\langle \ F \ \rangle = \int_{\Omega_0 } \biggl( \int_\Omega F(\omega) \P_{\xi} (d\omega) \biggr)\, \tilde \P( d\xi) \ \ \ \ \forall F \in L^1(\Omega, \F, \P).
\end{align}
Furthermore, there exists a measurable map
\begin{align}
 \Pi: (\Omega, \I) \to (\Omega_0, \I_0) 
\end{align}
 such that for every $F \in L^1(\Omega, \F, \P)$ the conditional expectation $\langle F \, | \, \I \rangle$ may be identified with a random variable in $L^1(\Omega_0, \I_0, \tilde \P)$ via the relation 
 \begin{align}\label{integration.disintegrated}
\langle F \, | \, \I \rangle(\omega) = \int_\Omega F(\tilde \omega) P_{\Pi(\omega)}(d\tilde\omega) \ \ \ \text{ for $\P$-almost every $\omega \in \Omega$.}
\end{align}
\end{lem}

\medskip

The next corollary relies on the previous decomposition of $\P$ to show that the random homogenized matrix of Theorem \ref{t.main} may be obtained by fixing the element
$\xi \in \Omega_0$ and applying on the probability space $(\Omega, \F, \P_\xi)$ the standard homogenization results \cite{GNO, Koslov79, PapVar} for stationary and ergodic
measures. For $\xi \in \Omega_0$ fixed, let indeed $a_{\mathrm{h}, \xi}$ be the deterministic homogenized matrix obtained by means of classical homogenization 
and defined as 
\begin{align}\label{fiber.homogenized.coefficient}
e_i \cdot a_{\mathrm{h}, \xi} e_j = \int_\Omega (e_i + \nabla \phi_{\xi,i}(0,\omega)) \cdot a(\omega) (e_j  + \nabla \phi_{\xi,j}(0, \omega)) \P_{\xi}(d\omega),
\end{align}
with $\phi_{i,\xi}$ the correctors solving \eqref{corrector.eq} with respect to $(\Omega, \F, \P_\xi)$. Then:

\smallskip

\begin{cor}\label{cor.abstract}
Let $\ah$ be the homogenized matrix introduced in Theorem \ref{t.main} and let $\Pi$ be the projection map of Lemma \ref{ergodic.disintegration}. 
Then, for $\tilde\P$-almost every $\xi \in \Omega_0$ and all $\omega \in \Pi^{-1}(\xi)$ we have
\begin{align}\label{abstract.reformulation}
\ah(\omega) = a_{\mathrm{h}, \xi}.
\end{align}
Therefore, for all $B \in \B(\mathcal{M}_{d, \textrm{sym}})$ we have
\begin{align}
 \P( \{ \omega \, \colon \, \ah(\omega) \in B \}) = \tilde\P(\{ \xi \, \colon \,  a_{\mathrm{h}, \xi} \in B \}).
\end{align}

\end{cor}

As we show in the next section, this abstract result admits a more explicit formulation in the case of coefficients being generated by a stationary Gaussian field. 

\medskip

\begin{rem}\label{convex.comb}{\bf Convex combination of stationary measures. } 
Lemma \ref{Ergodic.disintegration} is a generalization of the fact that stationarity is closed under convex combination and that ergodic measures are
extremal points of any convex set. More precisely, let $\{ \P_i \}_{i=1}^N$ be $N < +\infty$ be distinct stationary and ergodic probability measures
on $(\Omega, \F)$, with $\F$ countably generated. Let $\P$ be the  measure obtained as the convex combination: 
\begin{align}\label{convex.P}
\P = \sum_{k=1}^N \alpha_k \P_k, \ \ \ \ \ \ \sum_{k=1}^N \alpha_k = 1, \ \ 0\leq \alpha_k \leq 1.
\end{align}
It is easy to check that $\P$ is a stationary probability measure. Moreover, $\P$ is ergodic if and only if it is an extremal point, i.e. there exists 
$\alpha_k= 1$ for some $k \in \{1, \cdots, N\}$. 

\smallskip

We argue the only non trivial implication of the previous statement: Let us assume that $\P$ is ergodic. We show that if there exists $i \in \{1, \cdots, N \}$ such that $\alpha_{i} \in (0, 1)$, 
then the ergodicity property is contradicted. Indeed, by the last two conditions in \eqref{convex.P}, the previous assumption implies that there exists another
$j \in \{1, \cdots, N\}$, $j\neq i$, such that $\alpha_j \in (0,1)$. Moreover, since  all the measures are distinct, we may find
a set $B \in \F$ such that $\P_i(B)>0$ and $\P_i(B) \neq \P_j(B)$. By Birkhoff's theorem and the assumption on the ergodicity of each measure $\P_k$, it follows that the 
set
\begin{align}
A = \{ \omega \in \Omega \, \colon \,  \lim_{R \uparrow +\infty} \fint_{|x|< R} T_x\1_{B}(\omega) = \P_i(B)\}
\end{align}
satisfies $\P_{i}(A)=1$, $\P_j(A)=0$. Note that the set $A \in \I$. Hence, we have by \eqref{convex.P} that
$$
\P(A) = \alpha_i + \sum_{k =1, \atop k \neq i,j}^N \alpha_k \P_k(A). 
$$
Since we assumed that $\alpha_i, \alpha_j \in (0,1)$, $\P_k$ are probability measures and $\sum_{k=1}^N \alpha_k=1$, we infer that $\P(A) \in (0,1)$. 
This yields a contradiction.

\smallskip

Similarly, we note that since $\F$ is countably generated, i.e. $\F= \sigma(\{ B_n\}_{n\in \N})$), also the sets
\begin{align}
C_i = \bigcap_{n\in \N} \{ \omega \in \Omega \, \colon \,  \lim_{R \uparrow +\infty} \fint_{|x|< R} T_x\1_{B_n}(\omega) = \P_i(B_n)\}
\end{align}
satisfy $\P_{i}(C_j) = \delta_{ij}$ for all $i,j = 1,\cdots, N$. In particular, $\{ C_i\}_{i=1}^N$ they provide a $\P$-essential partition for $\Omega$ in each
one of which the limit of the spacial averages are given by integration in $\P_i$. Hence, in this easy case the set $\Omega_0$ of Lemma \ref{Ergodic.disintegration} is just
$\Omega_0 = \{1, \cdots, N\}$, the family of ergodic probabilities is $\{ P_i \}_{i=1}^N$ and $\tilde \P$ is 
the measure on $\I= \sigma(\{1, \cdots, N \})$ (uniquely) defined by $\tilde\P( i )= \alpha_i$ for $i= 1, \cdots, N$.
\end{rem}

\subsection{Application to stationary Gaussian fields}
The results of this section rely on \cite{SlezakGaussian}[Theorem 5 and Theorem 6]. For $d \geq 2$, $n \geq 1$, let $X$ be a stationary $\R^n$-valued Gaussian field on
$\Rd$ having continuous trajectories, i.e. the space of trajectories is given by $\Omega = C^0(\Rd, \R^n)$ with $\F$ the $\sigma$-algebra of the cylindrical sets. We assume that the group of transformations
$\{ \tau_x \}_{x\in \Rd}$ acts on each trajectory in $\Omega$ as $\tau_x X(\cdot) = X(\cdot + x)$. {With this choice of $\F$ and $\Omega$ condition \eqref{stochastic.continuity} is satisfied.}

\smallskip

Let $X$ be centered. We recall that for a given Gaussian field $X$, the autocorrelation matrix is given by
\begin{align}
 C(x) := \langle X(x) \otimes X(0) \rangle \in \R^{n \times n}, \ \ \ x \in \Rd.
\end{align}
Note that by stationarity we have that for all $x, y \in \Rd$ we have 
\begin{align}\label{symmetry}
\langle X(x) \otimes X(y) \rangle= C(x-y),\ \ \  \ \ C(x) = C^t(-x).
\end{align}
For $C \in C^0(\Rd)$, by Bochner's theorem \cite{yosida1995functional}[Chapter XI, Section 14] we may write
\begin{align}
C(x) = \frac{1}{(2\pi)^{\frac d 2}}\int_{\Rd} e^{-i x \cdot \xi} \hat C(\xi) \, d\xi,
\end{align}
where $\hat C(\xi)$, usually known as {\it spectral measure}, is a positive definite $\C^{n\times n}$-valued measure on $\Rd$. We remark that by \eqref{symmetry} it is 
easy to check that $\hat C$ satisfies
\begin{align}\label{spectral.measure.symmetry}
\hat C(\xi) = \hat C(\xi)^*, \ \ \ \ \hat C(\xi) = \hat C(-\xi)^t.
\end{align}

\smallskip

In the case of stationary Gaussian field, the ergodicity of the process $X$ under the translation group 
$\{\tau_x \}_{x\in \Rd}$ is equivalent to requiring that spectral measure $\hat C$ does not have an atomic part \cite{blum1972, eisenberg1972}. 
This and \eqref{spectral.measure.symmetry} yield that for any stationary Gaussian field, the non-ergodic behaviour is related to the presence in $\hat C$ of linear 
combinations of the form
\begin{align}\label{atomic.part}
\alpha_0 \delta_0 + \sum_{i=1}^{N} \alpha_i \delta_{-\omega_i} + \alpha_i^t \delta_{\omega_i}, 
\end{align}
for a positive-definite $\alpha_0 \in \R^{n \times n}$, hermitian matrices $\{\alpha_i\}_{i=1}^{N} \subset \C^{n \times n}$ and $\{\omega_i \}_{i=1}^{N} \subset \R_+^d$. 
We remark that the terms in the sum above may also be infinite, i.e. $N  = +\infty$, but from now on we restrict ourselves to the case $N \in \N$.

\smallskip

The special structure of Gaussian fields allows to extract a more explicit formulation for the $\sigma$-algebra $\I$ of the invariant sets.  As we show in the proof of
the next statement, the presence of the non-zero atoms in the spectral measure $\hat C$ corresponds to cosine terms in the process $X$. This yields that the
large-scale behaviour of $X$ and the $\sigma$-algebra $\I$ of the invariant sets crucially depend on possible resonances between the frequencies of oscillations. 
To this purpose, for any collection of values $\Omega=\{ \omega_i \}_{i=1}^N \subset \R_+$, with $1 \leq N< +\infty$, in \eqref{atomic.part},  we introduce the subset of
$\Z^N$ defined by
\begin{align}
\mathcal{R}_{\Omega}:= \{ k \in \Z^N \, \colon \, \sum_{i=1}^N k_i \omega_i = 0 \}.
\end{align}

\begin{rem}\label{commensurable.resonances}
If $\mathcal{R}_{\Omega}$ is non-trivial, then we may always write
\begin{align}\label{resonances.dimension}
\mathcal{R}_\Omega = \text{Span}_{\Z}(v^1, \cdots , v^r), 
\end{align}
for $1 \leq r \leq N-1$ and with $\{v^1, \cdots, v^r\}  \subset \Z^N \backslash{\{ {0}\}}$ satisfying the condition of linear integer-independence
\begin{align}
\sum_{j=1}^r m_j v^j = {0} \ \ \ \ \Leftrightarrow m_j = 0, \ \ \ \text{for all $j=1,\cdots, r$.} 
\end{align}
 This results follows from the classical theory of linear diophantine equations: In fact, up to a permutation of the elements in $\Omega$,
 we may always assume that there exists an index $1 \leq M \leq N$ such that the values $\omega_1, \cdots, \omega_M$ are all rationally incommensurable and, if $M < N$,
 that for all $j = M+1, \cdots, N$   we have $\omega_j = \sum_{i=1}^M q^j_i \omega_i$ for a unique $M$-tuple $(q_1^j, \cdots q_M^j )\subset \Q$.  By using this decomposition, solving 
 $ \sum_{i=1}^N k_i \omega_i = 0$ for $k \in \Z^N$ reduces to solving the system of $M$ equations with rational coefficients $k_i - \sum_{j=M+1}^Nq^j_i k_j =0$,
 $i=1,\cdots M$, for the $N$ integer variables $k_1, \cdots, k_N$. This system has at most $N - M$ linearly integer-independent solutions $v^1, \cdots v^{N-M} \in \Z^N$ 
 \cite{schrijver1998theory}[Chapter 4, Corollary 4.1c and formula (6)]. Since $N- M \leq N -1$, identity \eqref{resonances.dimension} is obtained.
\end{rem}

 \smallskip

For $F : \R^n \rightarrow \mathcal{M}_{d, \textrm{sym}}$ continuous and pointwise elliptic in the sense of \eqref{ellipticity}, we define 
\begin{align}\label{a.gaussian}
 A(X):= F \circ X(0), \ \ \ \ a(X,x) = F \circ X(x).
\end{align}
In the sake of a leaner notation, we state the following corollary in the special case $n=1$, i.e. when the Gaussian field $X$ is real-valued, and comment afterwards
on the generalization of this result to $n \geq 1$.
\begin{cor}\label{c.gaussian}
Let $X$ be a stationary, centered, Gaussian field having continuous correlation function $C$ and spectral measure $\hat C$ with an atomic part given by \eqref{atomic.part}
for $N < +\infty$. Let $a$ be defined as in 
\eqref{a.gaussian}. Then 

\smallskip

\begin{itemize}
\item[(a)] If $\mathcal{R}_\Omega= \{ 0\}$, i.e. the values $\{ \omega_i \}_{i=1}^N$ are rationally incommensurable, then for every $B \in \B(\mathcal{M}_{d, \textrm{sym}})$ we have
\begin{align}
\P(\{ \omega : \ah(\omega) \in B \}) = \tilde\P( \{(x, r) \in \R \times (\R_+)^N \, \colon \, a_{\mathrm{h}, (x, r)} \in B \}),
\end{align}
with
\begin{align}
\tilde \P(dx, dr) = \frac{e^{-\frac{|x|^2}{2\alpha_0^2}}}{\sqrt{2\pi \alpha_0^2}} \,  \prod_{i=1}^N \frac{r_i e^{-\frac{r^2_i}{\alpha_i^2}}}{\alpha_i^2}  \ dx \, dr_1 \cdots dr_N,
\end{align}
where $\{\alpha_i\}_{i=0}^N$ are the amplitudes in \eqref{atomic.part}. In other words, $\tilde\P$ is the probability measure associated to an independent Gaussian random 
variable and $N$ independent Rayleigh random variables.

\smallskip

\item[(b)] If, otherwise, $r \geq 1$ in \eqref{resonances.dimension}, then for every $B \in \B(\mathcal{M}_{d, \textrm{sym}})$ we have as well
\begin{align}
\P(\{ \omega : \ah(\omega) \in B \}) = \tilde\P( \{(x, r, \eta) \in \R \times (\R_+)^N \times \R^r \, \colon \, a_{\mathrm{h}, (x, r, \eta)} \in B \}).
\end{align}
Here,
\begin{align}
\tilde \P( dx, dr, d\eta) = \tilde\P_1(dx, dr) \tilde\P_2(d\eta)
\end{align}
with $\tilde \P_1$ as in case (a) and $\tilde\P_2$ the probability associated to the vector $\eta = \{ \eta_1, \cdots, \eta_r\}$ obtained for each $j=1, \cdots, r$ as
\begin{align}
\eta_j = \sum_{i=1}^N v^j_i \phi_i \ \ \ \ \ \ \text{ mod$(2\pi)$},
\end{align}
for $v^j$ as in \eqref{resonances.dimension} and $\{ \phi_i\}_{i=1}^N$ independent random variables which are uniformly distributed on $[0, 2\pi)$.
\end{itemize}
\end{cor}

The analogue of the previous result holds also in the higher-dimensional case $n \geq 1$, provided the random variables $x$, $ r_1, \cdots, r_N$ are
$\R^{n \times n}$-valued with each component independent and distributed as in the case $n=1$ above.

\smallskip

\section{Proofs}\label{s.proofs}
\begin{proof}[Proof of Theorem \ref{t.main}.] We resort to the proof of Theorem \ref{t.main} in the ergodic case \cite{GNO4, GNO, Koslov79} and show that only few 
modifications are needed in order to adapt it to our setting. 

\smallskip

Also in this case we rely on the construction of the sub-linear corrector $\phi= \{\phi_i\}_{i=1}^d$ satisfying \eqref{corrector.eq}. More precisely,
for every $i=1, \cdots, d$, we construct a random variable $\chi_i \in [L^2(\Omega, \F, \P)]^d$ which satisfies 
\begin{align}\label{zero.average}
\langle \chi_i \, | \,  \I \, \rangle = \langle \, \chi_i \, \rangle = 0
\end{align}
and such that
\begin{align}\label{gradient.corrector}
\nabla \phi_i(\omega, x) &= \chi_i(\tau_x \omega),
\end{align}
with $\phi_i$ solving \eqref{corrector.eq} for $\P$ almost every $\omega \in \Omega$.

\smallskip

To prove the existence of $\chi$ as above,  we modify the argument of \cite{GNO4} and enumerate below only the (few) steps which require
a non-trivial adaptation to our setting. Let $i =1, \cdots, d$ be fixed and let us write $\phi$ instead of $\phi_i$. Moreover, since  no 
ambiguity on the measure $\P$ considered occurs, we write $L^p(\Omega)$ instead of $L^p(\Omega, \F, \P)$.  For any $F \in L^2(\Omega)$ and $x \in \Rd$ 
let $T_x F:= F\circ \tau_x$. Thanks to \eqref{stationarity} and \eqref{stochastic.continuity}, the group of transformations $\{T_x \}_{x \in \Rd}$ provides a unitary and 
strongly continuous group of operators on $L^2(\Omega)$.  We may thus denote by $D_j$, $j=1, \cdots, d$, the infinitesimal generators of $T_{x \cdot e_j}$ 
{\cite{reed1981functional}[Subsection VIII.4]}, namely
\begin{align}\label{generator}
\lim_{h \downarrow 0^+}\frac{ T_{h e_j} - I}{h}= D_j \ \ \ \text{ in $L^2(\Omega)$.}
\end{align}
We denote by $\mathcal{D}:= \cap_{j=1}^d \mathcal{D}(D_j) \subset L^2(\Omega)$ the domain of the operator $D:=( D_1, \cdots , D_d)$ and note that, again by
\eqref{stochastic.continuity}, this set is dense in $L^2(\Omega)$.

\smallskip

We set
\begin{align}
&U := \{ \xi \in L^2(\Omega) \ \colon \ T_x \xi = \xi \ \forall x \in \Rd \},\\
&V(\Omega) :=  \overline{( \{  D\xi \, \colon \, \xi \in \mathcal{D} \})}^{L^2(\Omega)}.
\end{align}
Then, since
\begin{align}
U =  \{ \xi \in L^2(\Omega) \ \colon \ D \xi = 0 \text{ \ in $L^2(\Omega)$} \},
\end{align}
it follows that
\begin{align}\label{subspace.corrector}
V(\Omega) \subset (U^\perp)^d.
\end{align}
and that any element $\Psi \in V(\Omega)$ satisfies
\begin{align}
\langle \Psi \rangle = \langle \Psi \, | \, \I \, \rangle = 0.
\end{align}
Therefore, we define $\Psi \in V(\Omega)$ as the Lax-Milgram solution of 
\begin{align}\label{Lax.Milgram}
\langle \Psi \cdot A \chi \rangle = \langle \Psi \cdot A e_i \rangle,  \ \ \ \ \ \forall  \Psi \in V(\Omega),
\end{align}
with $A \in L^\infty(\Omega, \F, \P)$ as in \eqref{coefficient.field}. 

\smallskip

With this definition of $\chi \in V(\Omega)$ as the solution of \eqref{Lax.Milgram}, the same arguments used in \cite{GNO} yield \eqref{gradient.corrector} and  the first
line of \eqref{corrector.eq}. To conclude the proof of \eqref{corrector.eq}, we first observe that by $\chi \in L^2(\Omega)$ and the first identity in \eqref{gradient.corrector}, Neumann's ergodic theorem \cite{kre}[Theorem 1.4] yields  also that for almost every $\omega \in \Omega$
\begin{align}
\lim_{R \uparrow +\infty} \langle | \fint_{|x| < R}\nabla \phi(\omega, x) |^2 \rangle = 0.
\end{align}
From this identity we may argue exactly as in \cite{GNO4}[Proof of Corollary 1] and obtain also the last sub-linearity property in  \eqref{corrector.eq}.

\medskip

Equipped with the correctors $\{\phi_i \}_{i=1}^d$ as above, we argue as in the ergodic case to show Theorem \ref{t.main}: By \eqref{gradient.corrector}
and Birkhoff's ergodic theorem we have indeed that for almost every $\omega \in \Omega$ and every $R > 0$
$$
\lim_{\eps \downarrow 0}\int_{|x| < R} |\nabla \phi_i(\omega, \frac{x}{\eps})|^2 =\langle |\chi |^2 \, | \, \I \rangle.
$$
Furthermore, another application of Birkhoff's ergodic theorem together with a standard separability argument implies that for
almost every $\omega \in \Omega$ and every $\rho \in C^\infty_0(\Rd)$
\begin{align}
\lim_{\eps \downarrow 0}\int \rho(x)\nabla \phi_i(\omega, \frac x \eps) = \langle \chi_i \, | \, \I \, \rangle \int \rho(x) dx \stackrel{\eqref{zero.average}}{=} 0. 
\end{align}
These two limits yield for the whole family $\eps \downarrow 0^+$
 \begin{align}\label{gradients.to.zero}
 \nabla \phi_i(\omega, \frac \cdot \eps ) \rightharpoonup 0 \ \ \ \text{in $L^2_{loc}(\Rd)$.}
 \end{align}
Hence, for any bounded domain $D \subset B_R$, for some $R> 0$, the functions 
$$
w_i^\eps(\omega, x) := x_i + \eps\bigr(\phi_i(\omega, \frac{x}{\eps}) - \fint_{|x|< R}\phi_i( \omega, \frac{y}{\eps}) \, dy \bigr)
$$
satisfy for almost every $\omega \in \Omega$
\begin{align}\label{functions.converge}
w^\eps_i(\omega, \cdot) \rightharpoonup x_i \ \ \text{in $H^1(D)$.}
\end{align}
By \eqref{gradient.corrector} and the stationarity of $a$, we may argue similarly to obtain that for $\P$-almost every $\omega \in \Omega$
\begin{align}\label{fluxes.converge}
e_j \cdot a(\omega, \frac{x}{\eps}) \nabla w^\eps_i \rightharpoonup A_{\mathrm{h},ij}(\omega) \ \ \text{in $L^2_{loc}(\Rd)$.} 
\end{align}
We remark that the identification of the above limit with ${A}_{\mathrm{h},ij}$ as in \eqref{hom.conditional} follows by
$$
\langle \chi_j \cdot A ( e_i + \chi_i) \, | \, \I \, \rangle = 0.
$$
This identity is implied in turn by Birkhoff's ergodic theorem, \eqref{corrector.eq} and the bounds \eqref{ellipticity} for $a$ after taking the limit $R\uparrow +\infty$ in the estimate 
$$
| \fint_{|x| < R} \nabla \phi_j \cdot a (e_i + \phi_i) | \leq C R^{-1}\biggl( \fint_{|x|< 2R} |\phi_j - \fint_{|x|< 2R} \phi_j |^2 \biggl)^{\frac 12} \biggl( 1 +  \fint |\chi_i |^2 \biggr)^{\frac 1 2}.
$$
Here, the constant $C=C(d) < +\infty$. This estimate in turn easily follows by testing equation \eqref{corrector.eq} for $\phi_i$ with $\eta_R (\phi_j - \fint_{|x| < 2R} \phi_j)$, where $\eta_R$ is
 a cut-off function for $\{ |x| < R\}$ in $\{ |x| < 2R \}$.

\smallskip

Convergences \eqref{functions.converge} and \eqref{fluxes.converge} allow us to apply Tartar's Div-Curl lemma \cite{TartarBook2009}[Chapter 7, Lemma 7.2] as in the ergodic case and conclude the proof of Theorem \ref{t.main}.
\end{proof}

\begin{proof}[Proof of Lemma \ref{Ergodic.disintegration}]
We begin by constructing the family $\{ \P_{\xi} \}_{\xi \in \Omega_0}$ of stationary and ergodic measures on $(\Omega, \F)$: Let $\mathcal{S}$ be a countable collection of
sets generating $\F$. By Birkhoff's ergodic theorem, for $\P$-almost every $\omega \in \Omega$, we may define the probability measure $\P_\omega$ on $(\Omega, \F)$ as 
\begin{align}\label{asymptotic.prob}
\P_{\omega}(B) := \lim_{R \uparrow +\infty} \fint_{|x| < R} \1_{B}(\tau_x\omega) \, dx, \ \ \ B \in \mathcal{S}.
\end{align}
Since each probability measure is uniquely defined by its value on the generating set $\mathcal{S}$, it is immediate to check that $\P_{\omega}$ is stationary.
In addition, since if $I \in \I$ then the limit above exists for each $\omega \in \Omega$ and coincides with $\1_{\I}(\omega)$, from definition \eqref{asymptotic.prob}
it follows that
\begin{align}\label{ergodicity.fibers}
\P_\omega(I)= \1_{I}(\omega) \in \{ 0, 1 \}.
\end{align}
Equivalently, $\P_\omega$ is ergodic with respect to $\{\tau_x \}_{x\in \Rd}$.

\smallskip

Let $\Sigma_0 \in \F$ be the ($\P$-zero measure set) of elements $\omega \in \Omega$ for which $\P_\omega$ defined in \eqref{asymptotic.prob} does not exist. 
We introduce the equivalence relation on $\Omega$
\begin{align}\label{equivalence}
 \omega \sim \tilde \omega \ \ \Leftrightarrow  \ \ \P_\omega = \P_{\tilde \omega} \ \ \text{or} \ \ \omega, \tilde \omega \in \Sigma_0,
\end{align}
and define the quotient space  $\Omega_0 := \Omega / \sim$ and the projection operator 
\begin{align}\label{projection.map}
\Pi: \Omega \rightarrow  \Omega_0, \ \ \ \ \omega \mapsto \xi =\{ \tilde \omega \in \Omega \, \colon \, \tilde \omega \sim \omega \}.
\end{align}
Hence, thanks to \eqref{asymptotic.prob}, $\{ \P_{\omega} \}_{\omega \in \Omega} = \{ \P_\xi \}_{\xi \in \Omega_0}$ is a family of ergodic and stationary probability measures on $(\Omega, \F)$.
Rigorously, the probability $P_\xi$ corresponding to $\xi =\Pi(\Sigma_0) \in \Omega_0$ is not well-defined. However, since we take as measure 
$\tilde \P$ the push-forward $\P \circ \Pi^{-1}$, it follows that $\tilde \P (\Pi(\Sigma_0)) =0$  and thus that in the decomposition \eqref{ergodic.disintegration}
 the measure $\P_{\Pi(\Sigma_0)}$ is negligible.

\smallskip

We now define the $\sigma$-algebra $\I_0$ as the image of $\I$ under $\Pi$, i.e.
 $$
 \I_0:= \{ \Pi(I) \ \colon \ I \in \I \}, \ \ \ \ \Pi(I) := \{ \Pi(\omega) \, \colon \, \omega \in I \} \subset \Omega_0,
 $$ 
and argue that the above definition is well-posed and that $\I_0$ is a $\sigma$-algebra isomorphic to $\I$ in the sense that 
\begin{align}
I = \Pi^{-1} \circ \Pi (I),
\end{align}
 for every $I \in \I$. To do so, it suffices to observe that for every $I \in \I$ and $\xi \in \Omega_0$
\begin{align}\label{minimal.property}
\Pi^{-1}(\xi) \subset I \ \  \Leftrightarrow \ \ \ \Pi^{-1}(\xi)  \cap I = \emptyset.
\end{align}
The $\Rightarrow$ implication is trivial. For the $\Leftarrow$ implication we observe that whenever $\omega \in \Pi^{-1} (\xi) \cap I$, then  by 
\eqref{equivalence} and \eqref{ergodicity.fibers}  for every $\tilde \omega  \in \Pi^{-1}(\xi)$  we have that $\1_{I}(\tilde \omega) = \1_{I}( \omega ) = 1$.

\smallskip

From the previous argument and the fact that $\I \subset \F$, it follows that the map $\Pi$ is measurable from $(\Omega, \I)$ to $(\Omega_0 , \I_0)$, as well as from
$(\Omega , \F)$ to $(\Omega_0, \I_0)$. We define the probability measure $\tilde \P$ on $(\Omega_0, \I_0)$ as the push-forward of $\P$ under $\Pi$, i.e.
\begin{align}\label{P.tilde}
 \tilde \P = \P \circ \Pi^{-1}.
\end{align}

\smallskip

With these definitions of $\{ \P_\xi \}_{\xi \in \Omega_0}$ and $(\Omega_0, \I_0, \tilde\P)$, it remains to establish \eqref{ergodic.disintegration}, 
\eqref{integration.disintegrated}.  We begin with \eqref{integration.disintegrated} and use a standard approximation argument: Let $\F= \sigma(\mathcal{S})$.
For any $A \in \F$, by  Birkhoff's ergodic theorem we may construct for $\P$-almost every $\omega \in \Omega$ a probability measure $\hat{\P}_\omega$ on $(\Omega, \F)$ such that for all
$B \in \mathcal{S} \cup \{A \}$ it holds
$$
\hat{\P}_\omega(B)= \lim_{R\uparrow +\infty} \fint_{|x|<R} T_x 1_B(\omega) \, dx = \langle \1_B \, | \, \I \, \rangle.
$$
Since $\hat\P_\omega$ and $\P_\omega$ coincide on the set of generators $\mathcal{S}$, it follows by uniqueness that $\P_{\omega}(A) = \langle A \ | \, \I \, \rangle$ for 
$\P$-almost every $\omega \in \Omega$. Therefore,
\begin{align}
\langle \1_A \rangle = \int_{\Omega} \langle A \, | \, \I \, \rangle P(d\omega) = \int_{\Omega} \P_{\omega}(A)   P(d\omega).
\end{align}
By arguing similarly and using \eqref{stochastic.continuity}, for every $F \in L^1(\Omega, \F, \P)$ we have
\begin{align}
\langle F \rangle = \int_{\Omega} ( \int_\Omega F(\tilde \omega) \, \P_{\Pi(\omega)}(d\tilde\omega)  )\, \P(d\omega).
\end{align}
We now appeal to the definitions \eqref{projection.map} and \eqref{P.tilde} to conclude that
\begin{align}
\langle F \rangle = \int_{\Omega_0} \int_\Omega F(\omega) \P_\xi(d\omega) \, \tilde \P( d\xi),
\end{align}
i.e. formula \eqref{ergodic.disintegration}. The proof of this lemma is complete.
\end{proof}

\smallskip

\begin{proof}[Proof of Corollary \ref{cor.abstract}]
By \eqref{hom.conditional} of Theorem \ref{t.main} and \eqref{integration.disintegrated} of Lemma \ref{Ergodic.disintegration}, we may rewrite for $\P$- almost every 
$\xi \in \Omega_0$ and $\omega \in \Omega$ with $\Pi(\omega)=\xi$
\begin{align}\label{Ah.1}
e_i \cdot \ah(\omega) e_j = e_i \cdot \ah(\xi) e_j = \int_{\Omega_0} ( e_i + \nabla\phi_i(\omega, 0) ) \cdot a(0) ( e_j + \nabla \phi_j(\omega, 0) ) \P_\xi( d\omega).  
\end{align}
It thus remains to show that in the right-hand side above we may substitute the random variables $\nabla \phi_{i} , \nabla \phi_{j}$ with 
$\nabla \phi_{\xi,i} , \nabla \phi_{\xi, j}$. To do so, we resort to the construction of $\phi$ obtained in the proof of Theorem \ref{t.main} via the random variable $\chi \in L^2(\Omega, \F, \P)$
 (see \eqref{gradient.corrector}). We also remark that the same holds for $\phi_\xi$, where $\nabla \phi_\xi = \chi_\xi$ with $\chi_\xi \in L^2(\Omega, \F, \P_\xi)$. This either follows directly from the homogenization results for ergodic measures \cite{GNO}[Chapter 6, Section 6.1], or  by the exact same argument used in the proof of Theorem \ref{t.main} for $\phi$.

\medskip

We fix an index $i = 1, \cdots, d$ and drop it in the notation for $\phi_i$. On the one hand, by \eqref{corrector.eq}, for $\P$-almost every $\omega \in \Omega$ we have that
$\phi(\omega, x)$ solves \eqref{corrector.eq}. We use \eqref{ergodic.disintegration} of Lemma \ref{Ergodic.disintegration} to infer that also for $\tilde \P$-almost every 
$\xi \in \Omega_0$ and $\P_\xi$-almost every $\omega \in \Omega$ the functions $\phi( \omega, \cdot)$ satisfy \eqref{corrector.eq}. On the other hand, 
by Lemma \ref{Ergodic.disintegration} for $\tilde\P$-almost every $\xi \in \Omega_0$ the probability measure $\P_\xi$ in $(\Omega, \F)$ is stationary and ergodic with 
respect to the translations $\{ \tau_x \}_{x\in\Rd}$. We thus appeal to the standard results in homogenization \cite{GNO4, GNO, Koslov79}, to infer that there exists a 
random field $\phi_\xi$, having stationary gradient, solving \eqref{corrector.eq} for $\P_\xi$-almost every $\omega \in \Omega$. Therefore, for $\tilde \P$-almost every
$\xi \in \Omega_0$ and $\P_\xi$-almost every $\omega \in \Omega$ we have that the difference $\phi(\omega, \cdot)-\phi_\xi(\omega, \cdot)$ satisfies
\begin{align}
 -\nabla \cdot a(\omega, x)\nabla(\phi(\omega, x) - \phi_\xi( \omega, x) )= 0 \ \ \ \ \ \text{ in $\Rd$.}
\end{align}
This, together with the sub-linearity condition of \eqref{corrector.eq} for both $\phi(\omega, \cdot)$ and $\phi_\xi(\omega, \cdot)$ implies 
\begin{align}\label{equality.Rd}
\nabla\phi(\omega, \cdot) = \nabla \phi_\xi( \omega, \cdot) \ \ \ \text{ in $L^2(\Rd, \Rd)$.}
\end{align}

\smallskip

We now appeal to \eqref{gradient.corrector} for both the gradients $\nabla \phi, \nabla \phi_\xi$ to write 
\begin{align}
 \nabla\phi(\omega, x) = \chi(\tau_x \omega), \ \ \ \nabla\phi_\xi(\omega, x)= \chi_\xi( \tau_x \omega)
\end{align}
for $\chi \in [ L^2(\Omega, \F, \P)]^d$ and $\chi_\xi \in [L^2(\Omega, \F, \P_\xi)]^d$. Note that again by \eqref{ergodic.disintegration}, we have that for $\tilde\P$-almost every $\xi \in \Omega_0$ the random variable $\chi \in [L^2(\Omega, \F, \P_\xi)]^d$. This, the above identities and \eqref{equality.Rd} imply that for all $\rho \in C^\infty_0(\Rd)$ and $\psi \in [L^2(\Omega, \F, \P_\xi)]^d$
\begin{align}
 \int_\Omega \psi(\omega) \cdot ( \int_{\Rd} \rho(x) \chi(\tau_x \omega) \, dx ) \, \P_\xi(d\omega) = \int_\Omega \psi(\omega) \cdot ( \int_{\Rd} \rho(x) \chi_\xi(\tau_x \omega) \, dx ) \, \P_\xi(d\omega).
\end{align}
By stationarity of the measure $\P_\xi$ this may be rewritten as
\begin{align}
\int_\Omega ( \int_{\Rd} \rho(x) \psi(\tau_{-x}\omega) \, dx ) \cdot \chi(\omega) \, \P_\xi(d\omega)
= \int_\Omega ( \int_{\Rd} \rho(x) \psi(\tau_{-x}\omega) \, dx ) \cdot \chi_\xi(\omega) \, \P_\xi(d\omega).
\end{align}
We now choose a sequence  $\phi_\eps = \eps^{-d} \hat\phi( \frac{\cdot}{\eps})$ with $\hat \phi \in C^\infty_0(B_1)$ a mollifier and
appeal to \eqref{stochastic.continuity} to conclude that for every $\psi \in L^2(\Omega, \F, \P_\xi)$
\begin{align}
\int_\Omega \psi(\omega) \cdot \chi(\omega) \, \P_\xi(d\omega) = \int_\Omega \psi(\omega) \cdot \chi_\xi(\omega) \, \P_\xi(d\omega).
\end{align}
In particular, by applying this identity twice, first with $\Psi = a\chi_\xi$ and secondly with $\Psi = a\chi$, we get that the right-hand side of \eqref{Ah.1} equals to the 
right-hand side of \eqref{abstract.reformulation} in Corollary \ref{cor.abstract}.

\end{proof}

\begin{proof}[Proof of Corollary \ref{c.gaussian}]  
Let us split the spectral measure into the two components
\begin{align}
\hat C (\xi) = \hat C_{c}(\xi) + \hat C_{a}(\xi),
\end{align}
with the (positive) measure $C_a$ being the purely atomic part \eqref{atomic.part}. From this decomposition it follows that also $X$ may be decomposed into the two independent processes $X= X_c + X_a$, having spectral measure $\hat C_c$ and $\hat C_a$, respectively.  Moreover, the process $X_c$ is ergodic since $\hat C_c$ does not contain atoms \cite{blum1972, eisenberg1972}.  By \eqref{atomic.part}, the correlation function of the process $X_a$ may be written as
\begin{align}
C(x) = C_0 + \sum_{j=1}^N \alpha_j \cos (\omega_j \cdot x),
\end{align}
which corresponds to the centred stationary Gaussian process
\begin{align}
X_a(x) = x_0+ \sum_{j=1}^{N} \bigl( \zeta_j \cos (\omega_j \cdot x) + \zeta'_j \sin(\omega_j \cdot x) \bigr) = x_0+ \sum_{j=1}^{N} \text{Re} \bigl( (\zeta_j + i\zeta'_j)e^{i\omega_j \cdot x} \bigl),
\end{align}
for the independent random variables $x_0 \sim N(0, \alpha_0)$, $\zeta_j , \zeta_j' \sim N(0, \alpha_j)$ for all $j=1, \cdots, N$. In particular, we remark that 
if we set  $\zeta_j + i \zeta_j = R_j e^{i \phi_j}$, then the above process may be also rewritten as
\begin{align}\label{atomic.process}
X_a(x) = x_0+ \sum_{j=1}^{N} R_j \cos( \omega_j \cdot x + \phi_j)
\end{align}
where  all $R_j$,$\phi_j$ are independent and $R_j \sim Ray(\alpha_j)$, $\phi_j \in U([0, 2\pi))$.

\medskip

 By relying on \eqref{atomic.process} and the decomposition for $X$, we appeal to \cite{SlezakGaussian}[Theorem 5 and Theorem 6] to identify the 
$\sigma$-algebra of the invariant sets $\I$ in terms of the random variables in \eqref{atomic.process}. This, together with Corollary \ref{cor.abstract}, 
concludes the proof.
\end{proof}

\bigskip

\section*{Acknowledgements}

The authors acknowledge support through the CRC 1060 (The Mathematics of Emergent Effects) that is funded through the German Science Foundation (DFG), and the Hausdorff
Center for Mathematics (HCM) at the University of Bonn.

\bibliographystyle{amsplain}

\begin{thebibliography}{10}

\bibitem{ArmstrongKuusiMourrat}
S.~Armstrong, T.~Kuusi, and J.-C. Mourrat, \emph{The additive structure of
  elliptic homogenization}, Invent. Math. \textbf{208} (2017), no.~3,
  999--1154.

\bibitem{ArmstrongMourrat}
S.~N. Armstrong and J.-C. Mourrat, \emph{Lipschitz regularity for elliptic
  equations with random coefficients}, Arch. Ration. Mech. Anal. \textbf{219}
  (2016), no.~1, 255--348.

\bibitem{ArmstrongSmart}
S.~N. Armstrong and C.~K. Smart, \emph{Quantitative stochastic homogenization
  of convex integral functionals}, Ann. Sci. \'Ec. Norm. Sup\'er. (4)
  \textbf{49} (2016), no.~2, 423--481.

\bibitem{BellaFehrmanFischerOtto}
P.~Bella, B.~Fehrman, J.~Fischer, and F.~Otto, \emph{Stochastic homogenization
  of linear elliptic equations: Higher-order error estimates in weak norms via
  second-order correctors}, SIAM J. Math. Anal., \textbf{49}(6), 4658-4703. 

\bibitem{BellaGiuntiOttoPCMI}
P.~Bella, A.~Giunti, and F.~Otto, \emph{{Quantitative stochastic
  homogenization: local control of homogenization error through corrector}}, Mathematics and Materials, Park City Mathematics Series, pp 299-327, 2017.

\bibitem{BellaGiuntiOtto2nd}
\bysame, \emph{Effective multipoles in random media}, Preprint arXiv:1708.07672 (2017).

\bibitem{blum1972}
J.~R. Blum and B.~Eisenberg, \emph{Conditions for metric transitivity for
  stationary gaussian processes on groups}, Ann. Math. Statist. \textbf{43}
  (1972), no.~5, 1737--1741.

\bibitem{DalMaso.nonergodic}
F.~Cagnetti, G.~Dal~Maso, L.~Scardia, and C.~Zeppieri, \emph{Stochastic
  homogenisation of free-discontinuity problems}, submitted, Preprint
  arXiv:1712.07272 (2017).

\bibitem{DuerinckxGloriaOtto}
M.~{Duerinckx}, A.~{Gloria}, and F.~{Otto}, \emph{{The structure of
  fluctuations in stochastic homogenization}},  Preprint arXiv:1602.01717
 (2016).
 
 \bibitem{DuerinckxGloriaOtto2}
\bysame, \emph{{Robustness of the pathwise structure of fluctuations in stochastic homogenization}},  Preprint arXiv:1807.11781 (2018).

\bibitem{eisenberg1972}
B.~Eisenberg, \emph{A note on metric transitivity for stationary gaussian
  processes on groups}, Ann. Math. Statist. \textbf{43} (1972), no.~2,
  683--687.

\bibitem{FischerOtto}
J.~{Fischer} and F.~{Otto}, \emph{A higher-order large-scale regularity theory
  for random elliptic operators}, Comm. Partial Differential Equations
  \textbf{41} (2016), no.~7, 1108--1148.


\bibitem{GNO4}
A.~Gloria, S.~Neukamm, and F.~Otto, \emph{{A regularity theory for random
  elliptic operators}}, Preprint arXiv:1409.2678 (2014).

\bibitem{GNO}
\bysame, \emph{Quantification of ergodicity in stochastic homogenization:
  optimal bounds via spectral gap on {G}lauber dynamics}, Invent. Math.
  \textbf{199} (2015), no.~2, 455--515. \MR{3302119}

\bibitem{Gray.Probability2001}
R.~M. Gray, \emph{Probability, random processes, and ergodic properties},
  Springer US, 2001.


\bibitem{Gu_highorder}
Y.~Gu, \emph{High-order correctors and two-scale expansion in stochastic
  homogenization},  (2016), Preprint arXiv:1601.07958 (2016).


\bibitem{gumourrat-fluctuations}
Y.~{Gu} and J.-C. {Mourrat}, \emph{Scaling limit of fluctuations in stochastic
  homogenization}, Multiscale Model. Simul. \textbf{14} (2016), no.~1,
  452--481. \MR{3477309}

\bibitem{Koslov79}
S~M Kozlov, \emph{Averaging of random operators}, Mathematics of the
  USSR-Sbornik \textbf{37} (1980), no.~2, 167.

\bibitem{kre}
Ulrich Krengel, \emph{Ergodic theorems}, de Gruyter Studies in Mathematics,
  vol.~6, Walter de Gruyter \& Co., Berlin, 1985, with a supplement by Antoine
  Brunel.

\bibitem{MourratOtto}
J.-C. {Mourrat} and F.~{Otto}, \emph{Correlation structure of the corrector in
  stochastic homogenization}, Ann. Probab. \textbf{44} (2016), no.~5,
  3207--3233. \MR{3551195}

\bibitem{PapVar}
G.~C. Papanicolaou and S.~R.~S. Varadhan, \emph{Boundary value problems with
  rapidly oscillating random coefficients}, Random fields, {V}ol. {I}, {II}
  ({E}sztergom, 1979), Colloq. Math. Soc. J\'anos Bolyai, vol.~27,
  North-Holland, Amsterdam-New York, 1981, pp.~835--873.

\bibitem{reed1981functional}
M.~Reed and B.~Simon, \emph{I: Functional analysis}, Methods of Modern
  Mathematical Physics, Elsevier Science, 1981.  
  
\bibitem{schrijver1998theory}
Alexander Schrijver, \emph{Theory of linear and integer programming}, John
  Wiley \&amp; Sons, Inc., New York, NY, USA, 1986.

\bibitem{SlezakGaussian}
J.~{{\'S}l{\c e}zak}, \emph{{Asymptotic behaviour of time averages for
  non-ergodic Gaussian processes}}, Annals of Physics \textbf{383} (2017),
  285--311.

\bibitem{TartarBook2009}
L.~Tartar, \emph{The general theory of homogenization}, Lecture Notes of the
  Unione Matematica Italiana, vol.~7, Springer-Verlag, Berlin; UMI, Bologna,
  2009, A personalized introduction.

\bibitem{yosida1995functional}
K.~Yosida, \emph{Functional analysis}, Classics in Mathematics, Springer Berlin
  Heidelberg, 1995.

\end{thebibliography}

\end{document}